\documentclass[12pt, reqno]{amsart}
\usepackage{ amsmath,amsthm, amscd, amsfonts, amssymb, graphicx, color}
\usepackage[bookmarksnumbered, colorlinks, plainpages]{hyperref}
\newtheorem{thm}{Theorem}[section]
\newtheorem{cor}[thm]{Corollary}
\newtheorem{lem}[thm]{Lemma}

\newtheorem{exam}[thm]{Example}
\numberwithin{equation}{section}

\begin{document}

\title{generalized Cline's formula for g-Drazin inverse in a ring}

\author{Huanyin Chen}
\author{Marjan Sheibani Abdolyousefi}
\address{
Department of Mathematics\\ Hangzhou Normal University\\ Hang -zhou, China}
\email{<huanyinchen@aliyun.com>}
\address{Women's University of Semnan (Farzanegan), Semnan, Iran}
\email{<sheibani@fgusem.ac.ir>}

\subjclass[2010]{15A09, 47A11.} \keywords{Cline's formula; Generalized Drazin inverse; Drazin inverse; group inverse; Banach algebra.}

\begin{abstract} In this paper, we give a generalized Cline's formula for the generalized Drazin inverse. Let $R$ be a ring, and let $a,b,c,d\in R$ satisfying $$\begin{array}{c}
(ac)^2=(db)(ac), (db)^2=(ac)(db);\\
b(ac)a=b(db)a, c(ac)d=c(db)d.
\end{array}$$
Then $ac\in R^{d}$ if and only if $bd\in R^{d}$. In this case,
$(bd)^{d}=b((ac)^{d})^2d.$ We also present generalized Cline's formulas for Drazin and group inverses.
Some weaker conditions in a Banach algebra are also investigated. These extend the main results of Cline's formula on g-Drazin inverse of Liao, Chen and Cui (Bull. Malays. Math. Soc., {\bf 37}(2014), 37-42), Lian and Zeng (Turk. J. Math., {\bf 40}(2016), 161-165) and Miller and Zguitti (Rend. Circ. Mat. Palermo, II. Ser., {\bf 67}(2018), 105-114). As an application, new common spectral property of bounded linear operators over Banach spaces are obtained.\end{abstract}
\maketitle

\section{Introduction}

Let $R$ be an associative ring with an identity. The commutant of $a\in R$ is defined by $comm(a)=\{x\in
R~|~xa=ax\}$. The double commutant of $a\in R$ is defined by $comm^2(a)=\{x\in R~|~xy=yx~\mbox{for all}~y\in comm(a)\}$. An element $a\in R$ has g-Drazin inverse (i.e., generalized Drazin inverse) in case there exists $b\in R$ such that $$b=bab, b\in comm^2(a), a-a^2b\in R^{qnil}.$$ The preceding $b$ is unique if exists, we denote it by $a^{d}$. Here, $R^{qnil}=\{a\in R~|~1+ax\in R^{-1}~\mbox{for
every}~x\in comm(a)\}$. For a Banach algebra $\mathcal{A}$ it is well known
 that $$a\in \mathcal{A}^{qnil}\Leftrightarrow
\lim\limits_{n\to\infty}\parallel a^n\parallel^{\frac{1}{n}}=0\Leftrightarrow 1-\lambda a\in \mathcal{A}^{-1}~\mbox{for any scarlar}~\lambda.$$

Let $a,b\in R$. The Cline's formula for g-Drazin inverse stated that $ab\in R^d$ if and only if $ba\in R^d$ ad $(ba)^d=b[(ab)^d]^2a$ (see~\cite[Theorem 2.1]{LC}). Cline's formula plays an important role in the generalized inverse of matrix and operator theory (~\cite{G,KT,Mi,S,Y,Z1,Z2}). In ~\cite[Theorem 2.3]{L}, Lian and Zeng proved the generalized Cline's formula to the case when $aba=aca$. In ~\cite[Theorem 3.2]{Mi}, Miller and Zguitti generalized the Cline's formula under the condition
$acd=dbd$ and $dba=aca$. The motivation of this paper is to extend the Cline's formula for g-Drazin inverse to a wider case.

In Section 2, we present a new generalized Cline's formular for g-Draziin inverse. We also prove generalized Cline's formulas for Drazin and group inverses.
Let $R$ be a ring, and let $a,b,c,d\in R$ satisfying $$\begin{array}{c}
(ac)^2=(db)(ac), (db)^2=(ac)(db);\\
b(ac)a=b(db)a, c(ac)d=c(db)d.
\end{array}$$
Then $ac\in R^{d}$ if and only if $bd\in R^{d}$. In this case,
$(bd)^{d}=b((ac)^{d})^2d.$ This improves the main results of Cline's formula on g-Drazin inverse of Liao, Chen and Cui (~\cite[Theorem 2.1]{LC}), Lian and Zeng (~\cite[Theorem 2.3]{L}) and Miller and Zguitti (~\cite[Theorem 3.2]{Mi}). In Section 3, we investigate some weaker conditions in a Banach algebra under which the generalized Cline's formula holds. We prove that the preceding condition "$b(ac)a=b(db)a, c(ac)d=c(db)d$" can be dropped in a Banach algebra. Finally, in Section 4, we apply the generalized Cline's formula to common spectral property of bounded linear operators in a Banach space.

Throughout the paper, all rings are associative with an identity and all Banach algebras are complex. We use $\mathcal{A}^{rad}$ to denote the Jacobson radical of $\mathcal{A}$. The notations $\mathcal{A}^{d}, \mathcal{A}^{D}, \mathcal{A}^{\#}$ and $\mathcal{A}^{\ddag}$ stand for the sets of all g-Drazin, Drazin, group and p-Drazin invertible elements, respectively.

\section{generalized Cline's Formula}

For any elements $a,b$ in a ring $R$, it is well known that $ab\in R^{qnil}$ if and only if $ba\in R^{qnil}$ (see~\cite[Lemma 2.2]{L}). We start with the following generalization.

\begin{lem} Let $R$ be a ring, and let $a,b,c,d\in R$ satisfying $$\begin{array}{c}
(ac)^2=(db)(ac), (db)^2=(ac)(db);\\
b(ac)a=b(db)a, c(ac)d=c(db)d.
\end{array}$$ Then $ac\in R^{qnil}$ if and only if $bd\in R^{qnil}$.\end{lem}
\begin{proof} $\Longrightarrow$ Let $x\in comm(bd)$. Then we check that
$$\begin{array}{lll}
(dbdx^5bdbac)ac&=&dbdx^5bd(baca)c\\
&=&dbdx^5bd(bdba)c\\
&=&dbdx^5b(dbdba)c\\
&=&(dbdbd)x^5bdbac\\
&=&(dbdb)dx^5bdbac\\
&=&(acdb)dx^5bdbac\\
&=&(acdbd)x^5bdbac\\
&=&ac(dbdx^5bdbac)
\end{array}$$
Hence, $dbdx^5bdbac\in comm(ac)$, and so $1-dbd(x^5bdbacac)=1-(dbdx^5bdbac)ac\in R^{-1}$. By using Jacobson's Lemma (see~\cite[Lemma 2.1]{L}),
we see that $$\begin{array}{lll}
1-x^5bdbdbdbdbd&=&1-(x^5bdbdb(dbdb)d\\
&=&1-(x^5bdbdb(acdb)d\\
&=&1-(x^5bdbdbac)dbd\\
&=&1-(x^5bdbacac)dbd\\
&\in& R^{-1}.
\end{array}$$ Then
$$\begin{array}{ll}
&(1-xbd)(1+xbd+x^2bdbd+x^3bdbdbd+x^4bdbdbdbd)\\
=&(1+xbd+x^2bdbd+x^3bdbdbd+x^4bdbdbdbd)(1-xbd)\\
=&1-x^5bdbdbdbdbd\\
\in& R^{-1},
\end{array}$$ and so $bd\in R^{qnil}$.

$\Longrightarrow$ Since $bd\in R^{qnil}$, by~\cite[Lemma 2.2]{L}, $db\in R^{qnil}$. Applying the preceding discussion, we see that $ca\in R^{qnil}$.
By using ~\cite[Lemma 2.2]{L} again, we have $ac\in R^{qnil}$, as desired.\end{proof}

We are now ready to prove:

\begin{thm} Let $R$ be a ring, and let $a,b,c,d\in R$ satisfying $$\begin{array}{c}
(ac)^2=(db)(ac), (db)^2=(ac)(db);\\
b(ac)a=b(db)a, c(ac)d=c(db)d.
\end{array}$$
Then $ac\in R^{d}$ if and only if $bd\in R^{d}$. In this case,
$$(bd)^{d}=b((ac)^{d})^2d.$$\end{thm}
\begin{proof} $\Longrightarrow $ Let $(ac)^{d}=h$, $e=bh^2d$.

Step 1. Let $t\in comm(bd)$. Then we check that
$$\begin{array}{lll}
ac(dtbdbdbac)&=&(acdbdbd)(tbac)\\
&=&(dbdbdbd)(tbac)\\
&=&dtbdb(dbdba)c\\
&=&(dtbdbdbac)ac.
\end{array}$$
Thus $dtbdbdbac\in comm(ac)$, and so $(dtbdbdbac)h=h(dtbdbdbac).$
It is easy to verify that $acdbd=a(cdbd)=a(cacd)=(ac)^2d=(dbac)d=dbacd$, and so $(bacd)(bd)=(bd)(bacd)$,
Then we compute that
 $$\begin{array}{lll}
 et&=&(bh^6(ac)^4d)t=bh^6(db)^3acdt=bh^6dbdbd(bacd)t\\
&=&bh^6d(bacd)bdbdt=bh^6db(acdb)dbdt=bh^6(dbdbdbdbd)t\\
&=&bh^6dtbdb(dbdb)d=bh^6dtbdba(cdbd)=bh^6dtbdb(acac)d\\
&=&bh^6(dtbdbdbac)d=b(dtbdbdbac)h^6d=tbdbdbdbach^6d\\
 &=&tb(ac)^4h^6d=tbh^2d=te,
 \end{array}$$
and then $e\in comm^2(bd)$.

Step 2. We directly verify that
 $$\begin{array}{lll}
 e(bd)e&=&bh^2dbdbh^2d=bh^2acdbh^2d\\
 &=&bh^2acdb(ac)^2h^4d=bh^2(ac)^4h^4d=b(ac)^4h^6d=e.
 \end{array}$$

 Step 3. Let $p=1-(ac)h$. Then $(pa)c=ac-(ac)^2h\in R^{qnil}$.
On easily checks that
  $$\begin{array}{lll}
  bd-(bd)^2e&=&bd-bdbdbh^2d\\
  &=&bd-b(dbdb)ach^3d\\
  &=&bd-b(acdb)ach^3d\\
 &=&bd-bac(dbac)h^3d\\
  &=&bd-b(ac)^3h^3d\\
  &=&b(1-ach)d\\
  &=&b(pd).
  \end{array}$$
 We directly compute that
$$\begin{array}{rll}
(pac)^2&=&p(ac)^2p=p(dbac)p=(pdb)(pac),\\
(pdb)^2&=&pdbpdb=pdb[1-(1-p)]db\\
&=&p(db)^2-pdbac(ac)^ddb\\
&=&p(db)^2-p(ac)^2(ac)^ddb\\
&=&p(db)^2=pacdb=(pac)(pdb);\\
b(pac)pa&=&bpaca=baca-b(ac)^d(ac)^2a\\
&=&baca-b(ac)^dd(baca)=bdba-b(ac)^dd(bdba)\\
&=&bdba-b(ac)^d(acdb)a=b(pdb)a,\\
c(pac)pd&=&cacd-c(ac)^d(ac)^2d=cdbd-c(ac)^d(ac)^2d\\
&-&[c(ac)^da(cdbd)-c(ac)^da(cacd)]=cdbd-cac(ac)^dacd\\
&-&[c(ac)^dacdbd-c(ac)^dacdbac(ac)^dd]\\
&=&cdbd-cac(ac)^dacd-c(ac)^dacdb[1-ac(ac)^d]d\\
&=&cdbd-cdb(ac)^dacd-c(ac)^dacdbpd\\
&=&[cdb-c(ac)^dacdb]pd=c(pdb)pd.
\end{array}$$
Since $(pa)c=ac-(ac)^2(ac)^d\in R^{qnil}$. In view of Lemma 2.1, $b(pd)\in R^{qnil}$. Therefore $bd$ has g-Drazin inverse $e$ and $e=bh^2a=(bd)^{d}.$

$\Longleftarrow $ In view of~\cite[Theorem 2.1]{LC}, $db\in R^d$. Applying the preceding discussion, we have $ca\in R^d$.
By using~\cite[Theorem 2.1]{LC} again, $ac\in R^d$. This completes the proof.\end{proof}

In the case that $c=b$ and $d=a$, we recover the Cline's formula for g-Drazin inverse . In~\cite[Theorem 2.3]{L}, Lian and Zeng concerned with Cline's formula under the condition $aba=aca$. We now derive

\begin{cor} Let $R$ be a ring, and let $a,b,c\in R$ satisfying $$\begin{array}{c}
(aba)b=(aca)b, b(aba)=b(aca),\\
(aba)c=(aca)c, c(aba)=c(aca).
\end{array}$$ Then $ac\in R^{d}$ if and only if $ba\in R^{d}$ and $(ba)^{d}=b((ab)^{d})^2a$.\end{cor}\begin{proof} Choosing $d=a$ in Theorem 3.2, we obtain the result.\end{proof}

 \begin{cor} (~\cite[Theorem 3.2]{Mi}) Let $R$ be a ring, and let $a,b,c\in R$ satisfying $$acd=dbd, dba=aca.$$ Then $ac\in R^{d}$ if and only if
 $ba\in R^{d}$ and $(ba)^{d}=b((ab)^{d})^2a$.\end{cor}\begin{proof} By hypothesis, we easy to check that
 $$\begin{array}{c}
(ac)^2=(db)(ac), (db)^2=(ac)(db);\\
b(ac)a=b(db)a, c(ac)d=c(db)d,
\end{array}$$ thus completing the proof by Theorem 3.2.\end{proof}

 An element $a\in R$ has Drazin inverse in case there exists $b\in R$ such that $$b=bab, b\in comm^2(a), a^k=a^{k+1}b$$ for some $k\in {\Bbb N}$ The preceding $b$ is unique if it exists, we denote it by $a^{D}$. The smallest $k$ satisfying the preceding condition is called the Drazin index of $a$, and denote it by $i(a)$.

\begin{thm} Let $R$ be a ring, and let $a,b,c,d\in R$ satisfying $$\begin{array}{c}
(ac)^2=(db)(ac), (db)^2=(ac)(db);\\
b(ac)a=b(db)a, c(ac)d=c(db)d.
\end{array}$$ Then $ac\in R^{D}$ if and only if $bd\in R^{D}$. In this case,
$$\begin{array}{c}
(bd)^{D}=b((ac)^{D})^2d,\\
i(bd)\leq i(ac)+2.
\end{array}$$
\end{thm}
\begin{proof} $\Longrightarrow $ Since $ac\in R^D$, we see that $ac\in R^d$. It follows by Theorem 3.2 that $bd\in R^d$ and
$(bd)^d=b((ac)^d)^2d$. Then $bd(bd)^d=(bd)^dbd$ and $(bd)^d=(bd)^d(bd)(bd)^d$.
Clearly, we have
$$\begin{array}{lll}
1-(bd)^2((ac)^d)^2&=&1-acdb(ac)((ac)^d)^3\\
&=&1-ac(dbac)((ac)^d)^3\\
&=&1-ac(ac)^d.
\end{array}$$ Hence, $$\begin{array}{lll}
bd-(bd)^2(bd)^d&=&bd-(bd)^2b((ac)^d)^2d\\
&=&b[1-(db)^2((ac)^d)^2]d\\
&=&b[1-ac(ac)^d]d.
\end{array}$$
Further, we have $$\begin{array}{lll}
[1-(bd)(bd)^d](bd)^3&=&b[1-ac(ac)^d](dbdb)d\\
&=&b[1-ac(ac)^d](acdb)d\\
&=&b[1-(ac)(ac)^d](ac)dbd.
\end{array}$$ Write $m=i(ac)$. By induction, we have $$\begin{array}{lll}
[1-(bd)(bd)^d](bd)^{m+2}&=&b[1-(ac)(ac)^d](ac)^mdbd\\
&=&0.
\end{array}$$ Therefore $[bd-(bd)^2(bd)^d]^{m+2}=0$, and so $bd$ has Drazin inverse. Moreover, we have $i(bd)\leq i(ac)+2,$ as required.

$\Longrightarrow $ This is proved as in Theorem 3.2.\end{proof}

\begin{cor} Let $R$ be a ring, and let $a,b,c\in R$ satisfying $$\begin{array}{c}
(aba)b=(aca)b, b(aba)=b(aca),\\
(aba)c=(aca)c, c(aba)=c(aca).
\end{array}$$ Then $ac\in R^D$ if and only if
 $ba\in R^D$. In this case, $$\begin{array}{lll}
 (ba)^{D}&=&b((ac)^{D})^2a,\\
 i(ba)&\leq& i(ac)+1.
 \end{array}$$\end{cor}\begin{proof} We prove that $ac\in R^D$ if and only if
 $ba\in R^D$ by choosing $d=a$ in Theorem 2.5.
Moreover, we check that $ba-(ba)^2(ba)^D=b[1-ac(ac)^D]a$, and so $$\begin{array}{lll}
[1-(ba)(ba)^D](ba)^2&=&b[1-ac(ac)^D]aba\\
&=&baba-bac(ac)^Daba\\
&=&baca-bac((ac)^D)^2a(caba)\\
&=&baca-bac((ac)^D)^2a(caca)\\
&=&b[ac-(ac)^2(ac)^D]a\\
&=&b[1-(ac)(ac)^D](ac)a.
\end{array}$$ By induction, we have
$[1-(ba)(ba)^D](ba)^{m+1}=b[1-(ac)(ac)^D](ac)^ma=0$, where $m=i(ac)$.
This shows that $$(ba)^{m+1}-(ba)^{m+2}(ba)^D=[1-(ba)(ba)^d](ba)^{m+1}=0.$$ Therefore $i(ba)\leq i(ac)+1,$ as asserted.\end{proof}

The group inverse of $a\in R$ is the unique element $a^{\#}\in R$ which satisfies $aa^{\#}=a^{\#}a, a=aa^{\#}a, a^{\#}=a^{\#}aa^{\#}.$ We denote the set of all group invertible elements of $R$ by $R^{\#}$. As is well known, $a\in R^{\#}$ if and only if $a\in R^D$ and $i(a)=1$.

\begin{thm} Let $R$ be a ring, and let $a,b,c\in R$ satisfying $$\begin{array}{lll}
(aba)b=(aca)b, b(aba)=b(aca),\\
(aba)c=(aca)c, c(aba)=c(aca).
 \end{array}$$ If $ac\in R^{\#}$, then
$(ba)^2\in R^{\#}$. In this case, $(ac)^{\#}=a[(ba)^2]^{\#}c$.\end{thm}
\begin{proof} Since $ac\in \mathcal{A}^{\#}$, it follows by Corollary 2.6 that $ba\in \mathcal{A}^D$ and $(ba)^D=b[(ac)^2]^Da$.
Moreover, we have $i(ba)\leq i(ac)+1=2$. Set $x=(ba)^D$. Then $(ba)^2=(ba)^3x=(ba)^2x^2(ba)^2, x^2=x^2(ba)^2x^2$. Hence $[(ba)^2]^{\#}=x^2$.
We observe that
$$\begin{array}{lll}
a[(ba)^2]^Dc&=&a[(ba)^D]^2c\\
&=&ab[(ac)^2]^Dab[(ac)^2]^Dac\\
&=&ab[(ac)^2]^D(abac)(ac)^D]^2\\
&=&ab[(ac)^2]^D(ac)^2[(ac)^D]^2\\
&=&(abac)[(ac)^D]^3\\
&=&(ac)^2[(ac)^D]^3\\
&=&(ac)^D,
\end{array}$$ therefore $(ac)^{\#}=a[(ab)^2]^{\#}c$, as desired.\end{proof}

\begin{cor} Let $\mathcal{A}$ be a Banach algebra, and let $a,b,c\in \mathcal{A}$ satisfying $aba=aca$. If $ac\in \mathcal{A}^{\#}$, then
$(ab)^2\in \mathcal{A}^{\#}$. In this case, $(ac)^{\#}=a[(ab)^2]^{\#}c$.\end{cor}
\begin{proof} This is clear from Theorem 2.7.\end{proof}

\section{Extensions in Banach algebras}

In this section, we investigate the generalized Cline's formula in a Banach algebra. We observe that the condition "$b(ac)a=b(db)a, c(ac)d=c(db)d$" in Theorem 3.2 can be dropped in Banach algebra.

\begin{thm} Let $\mathcal{A}$ be a Banach algebra, and let $a,b,c,d\in \mathcal{A}$ satisfying $$\begin{array}{c}
(ac)^2=(db)(ac),\\
(db)^2=(ac)(db).
\end{array}$$ Then $ac\in \mathcal{A}^{d}$ if and only if $bd\in \mathcal{A}^{d}$. In this case, $(bd)^d=b[(ac)^d]^2d$.\end{thm}
\begin{proof} $\Longrightarrow $  Let $aca=a^{'}, b=b^{'}, c=c^{'}$ and $dbd=d^{'}$. Then we have
$$\begin{array}{lll}
(a'c')^2=(ac)^4=(db)^3ac==dbdbacac=(d'b')(a'c'),\\
(d'b')^2=(db)^4=(ac)^2(db)^2=(a'c')(d'b'),\\
b'(a'c')a'=b(ac)^3a=bdbdbaca=b'(d'b')a',\\
c'(a'c')d'=c(acac)dbd=c(db)^3d=c'(d'b')d'.
 \end{array}$$ Since $ac\in \mathcal{A}^{d}$, it follows by~\cite[Theorem 2.7]{J} that $a'c'=(ac)^2\in \mathcal{A}^d$, In light of Theorem 3.2,
 $b'd'=(bd)^2\in \mathcal{A}^d$. Therefore $bd\in \mathcal{A}^d$ by~\cite[Theorem 2.7]{J}. Moreover, we have
$$\begin{array}{lll}
(bd)^{d}&=&[(bd)^2]^{d}bd\\
&=&(b'd')^dbd=b'[(a'c')^d]^2d'bd\\
&=&b[(ac)^d]^4(db)^2d\\
&=&b[(ac)^d]^4(ac)^2d\\
&=&b[(ac)^d]^2d,
\end{array}$$ as required.

$\Longleftarrow $ Since $db\in \mathcal{A}^d$, applying the preceding discussion, we have $ca\in {\Bbb A}^d$.
Therefore $ac\in {\Bbb A}^d$, by using the Cline's formula.\end{proof}

As easy consequences, we now derive

\begin{cor} Let $\mathcal{A}$ be a Banach algebra, and let $a,b,c\in \mathcal{A}$. If $(aba)b=(aca)b, (aba)c=(aca)c$, then $ac\in \mathcal{A}^{d}$ if and only if $ba\in \mathcal{A}^{d}$. In this case, $(ba)^d=b[(ac)^d]^2a$.\end{cor}
\begin{proof} This is obvious by choosing $d=a$ in Theorem 3.1.\end{proof}

\begin{cor} Let $\mathcal{A}$ be a Banach algebra, and let $a,b,c,d\in \mathcal{A}$ satisfying $$\begin{array}{c}
(ac)^2=(db)(ac),\\
(db)^2=(ac)(db).
\end{array}$$ Then $ac\in \mathcal{A}^D$ if and only if $bd\in \mathcal{A}^D$. In this case, $(bd)^D=b[(ac)^D]^2d$.\end{cor}
\begin{proof}  $\Longrightarrow $  Since $ac\in \mathcal{A}^D$, we see that $ac\in \mathcal{A}^d$. By virtue of Theorem 3.1, $bd\in \mathcal{A}^{d}$ and  $(bd)^d=b[(ac)^d]^2d$. Let $m=i(ac)$. AS in the proof of Theorem 2.4, we verify that
$[bd-(bd)^2(bd)^d]^{m+2}=0,$ and so $bd\in \mathcal{A}^d$. Moreover, $(bd)^D=(bd)^d$, as required.

$\Longleftarrow $ This is proved as in Theorem 3.1.\end{proof}

An element $a$ in a Banach algebra $\mathcal{A}$ has p-Drazin inverse provided that there exists
$b\in comm(a)$ such that $b=b^2a, a^k-a^{k+1}b \in \mathcal{A}^{rad}$ for some $k\in {\Bbb N}$. The preceding $b$ is unique if exists, and we denote it by $a^{\ddag}$ (see~\cite{W}). We now derive

\begin{thm} Let $\mathcal{A}$ be a Banach algebra, and let $a,b,c,d\in \mathcal{A}$ satisfying $$\begin{array}{c}
(ac)^2=(db)(ac),\\
(db)^2=(ac)(db).
\end{array}$$ Then $ac\in \mathcal{A}^{\ddag}$ if and only if $bd\in \mathcal{A}^{\ddag}$. In this case, $(bd)^{\ddag}=b[(ac)^{\ddag}]^2d$.\end{thm}
\begin{proof} $\Longrightarrow $  Since $ac\in \mathcal{A}^{\ddag}$, we have $ac\in \mathcal{A}^d$. In light of Theorem 3.1, $bd\in \mathcal{A}^{d}$ and  $(bd)^d=b[(ac)^d]^2d$. One easily checks that
$$\begin{array}{lll}
1-(bd)(bd)^{d}&=&1-b(dbac)[(ac)^d]^3d\\
&=&1-b(ac)^2[(ac)^d]^3d\\
&=&1-b(ac)^dd.
\end{array}$$ Assume that $[ac-(ac)^2(ac)^{\ddag}]^k\in \mathcal{A}^{rad}$.
Then $$[1-(bd)(bd)^{d}](bd)^3=b[1-ac(ac)^d]ac(dbd).$$ By induction,
we have $$\begin{array}{lll}
[bd-(bd)^2(bd)^d]^{m+2}&=&[1-(bd)(bd)^{d}](bd)^{m+2}\\
&=&b[1-ac(ac)^d](ac)^m(dbd)\\
&=&b[ac-ac^2(ac)^d]^m(dbd)\\
&\in &\mathcal{A}^{rad}.
\end{array}$$ Therefore $$(bd)^{\ddag}=b[(ac)^{\ddag}]^2d,$$ as asserted.

$\Longleftarrow $ By virtue of~\cite[Theorem 3.6]{W}, $db\in \mathcal{A}^{\ddag}$. Then we have $ca\in \mathcal{A}^{\ddag}$ by the discussion above.
So the theorem is true by ~\cite[Theorem 3.6]{W}.\end{proof}

\begin{cor} Let $\mathcal{A}$ be a Banach algebra, and let $a,b,c\in \mathcal{A}$. If $(aba)b=(aca)b, (aba)c=(aca)c$, then $ac\in \mathcal{A}^{\ddag}$ if and only if $ba\in \mathcal{A}^{\ddag}$. In this case, $(ba)^{\ddag}=b[(ac)^{\ddag}]^2a$:\end{cor}
\begin{proof} This is obvious by choosing $d=a$ in Theorem 3.5.\end{proof}

The following example illustrates that Theorem 3.4 is not a trivial generalization of~\cite[Theorem 4.1]{Mi}.

The following example illustrates that Theorem 3.4 is not a trivial generalization of~\cite[Theorem 4.1]{Mi}.

\begin{exam}\end{exam} Let $\mathcal{A}=M_4({\Bbb C})$. Choose $$a=b=c=\left(
\begin{array}{cccc}
0&1&0&0\\
0&0&1&0\\
0&0&0&1\\
0&0&0&0
\end{array}
\right), d=\left(
\begin{array}{cccc}
0&2&0&0\\
0&0&1&0\\
0&0&0&1\\
0&0&0&0
\end{array}
\right)
\in M_4({\Bbb C}).$$ Then $$\begin{array}{c}
(ac)^2=(db)(ac)=0,\\
(db)^2=(ac)(db)=0.
\end{array}$$ We see that  $ac$ and $bd$  are nilpotent matrices and so have p-Drazin inverses. In this case, $$acd=\left(\begin{array}{cccc}
0&0&1&0\\
0&0&0&0\\
0&0&0&0\\
0&0&0&0
\end{array}
\right)\neq \left(\begin{array}{cccc}
0&0&2&0\\
0&0&0&0\\
0&0&0&0\\
0&0&0&0
\end{array}
\right)=dbd.$$

\section{Applications}

Let $X$ be  Banach space, and let $\mathcal{L}(X)$ denote the set of all bounded linear operators from Banach space to itself, and let $a\in \mathcal{L}(X)$. The Drazin spectrum $\sigma_D(a)$ and g-Drazin spectrum $\sigma_{d}(a)$ are defined by $$\begin{array}{c}
\sigma_D(a)=\{ \lambda\in {\Bbb C}~|~\lambda-a\not\in A^{D}\};\\
\sigma_{d}(a)=\{ \lambda\in {\Bbb C}~|~\lambda-a\not\in A^{d}\}.
\end{array}$$ For the further use, we now record the following.

\begin{lem} Let $R$ be a ring, and let $a,b,c,d\in R$ satisfying $$\begin{array}{c}
(ac)^2=(db)(ac), (db)^2=(ac)(db);\\
b(ac)a=b(db)a, c(ac)d=c(db)d.
\end{array}$$ Then $1-ac\in R^{-1}$ if and only if $1-bd\in R^{-1}$. In this case,
$$(1-bd)^{-1}=[1-b(1-ac)^{-1}(acd-dbd)][1+b(1-ac)^{-1}d].$$
\end{lem}\begin{proof} $\Longrightarrow$ Let $s=(1-ac)^{-1}$. Then
$s(1-ac)=1$, and so $1-s=-sac$. We check that
$$\begin{array}{lll}
(1+bsd)(1-bd)&=&1-b(1-s)d-bsdbd\\
&=&1+bsacd-bsdbd\\
&=&1+bs(acd-dbd).
\end{array}$$
Hence, $$\begin{array}{ll}
&[1-bs(acd-dbd)](1+bsd)(1-bd)\\
=&1-b(1-s)d-bsdbd\\
=&1+bsacd-bsdbd\\
=&1-bs(acd-dbd)bs(acd-dbd)\\
=&1-bs(acd-dbd)b(1+sac)(acd-dbd)\\
=&1-bs(ac-db)db(acd-dbd)\\
=&1.
\end{array}$$ Also we check that
$$\begin{array}{lll}
(1-bd)(1+bd+bacsd)&=&1-bdbd+b(1-db)acsd\\
&=&1-b[db(1-ac)-(1-db)ac]sd\\
&=&1-b(db-ac)sd;
\end{array}$$ hence, we have
$$\begin{array}{ll}
&(1-bd)(1+bd+bacsd)[1+b(db-ac)sd]\\
=&1-b(db-ac)sdb(db-ac)sd\\
=&1-b(db-ac)sdb(db-ac)(1+acs)d\\
=&1-b(db-ac)sdb(db-ac)d\\
=&1.
\end{array}$$ That is, $1-bd$ is right and left invertible. Therefore $$(1-bd)^{-1}=[1-bs(acd-dbd)](1+bsd),$$ as desired.

$\Longleftarrow$ In light of ~\cite[Theorem 2.1]{Mi}, $1-db\in R^{-1}$. Applying the discussion above, we see that $1-ca\in R^{-1}$.
By using~\cite[Theorem 2.1]{Mi} again, $1-ac\in R^{-1}$, as asserted.\end{proof}

We have at our disposal all the information necessary to prove the following.

\begin{thm} Let $A,B,C,D\in \mathcal{L}(X)$ such that $$\begin{array}{c}
(AC)^2=(DB)(AC), (DB)^2=(AC)(DB);\\
B(AC)A=B(DB)A, C(AC)D=A(DB)D.
\end{array}$$ then $$\sigma_d(BD)=\sigma_d(AC).$$\end{thm}
\begin{proof} Case 1. $0\in \sigma_d(BD)$. Then $BD\not\in A^{d}$. In view of Theorem 2.2, $AC\not\in A^{d}$. Thus $0\in \sigma_d(AC)$.

Case 2. $0\not\in \lambda\in\sigma_d(BD)$. Then $\lambda\in acc\sigma(BD)$; hence,
$$\lambda=\lim\limits_{n\to \infty}\{ \lambda_n ~|~ \lambda_n I-BD\not\in \mathcal{L}(X)^{-1}\}.$$
For $\lambda_n\neq 0$, we have $I-(\frac{1}{\lambda_n} A)C\in \mathcal{L}(X)^{-1}$. In light of Lemma 4.1, $I-B(\frac{1}{\lambda_n}D)\in \mathcal{L}(X)^{-1}$, and then
$$\lambda=\lim\limits_{n\to \infty}\{ \lambda_n ~|~ \lambda_n I-AC\not\in \mathcal{L}(X)^{-1}\}\in acc\sigma(AC)=\sigma_d(AC).$$
Therefore $\sigma_d(BD)\subseteq \sigma_d(AC).$ Analogously, we have $\sigma_d(AC)\subseteq \sigma_d(BD),$ the result follows.\end{proof}

\begin{cor} Let $A, B, C\in \mathcal{L}(X)$. If $(ABA)B=(ACA)B, (AB$ $A)C=(ACA)C$, then $$\sigma_d(AC)=\sigma_d(BA).$$\end{cor}
\begin{proof} By choosing $D=A$ in Theorem 4.2, we complete the proof.\end{proof}

For the Drazin spectrum $\sigma_D(a)$, we now derive

\begin{thm} Let $A,B,C,D\in \mathcal{L}(X)$ such that $$\begin{array}{c}
(AC)^2=(DB)(AC), (DB)^2=(AC)(DB);\\
B(AC)A=B(DB)A, C(AC)D=A(DB)D.
\end{array}$$ then $$\sigma_D(BD)=\sigma_D(AC).$$\end{thm}
\begin{proof} By virtue of Theorem 2.5, $AC\in \mathcal{L}(X)^{D}$ implies that $BD\in \mathcal{L}(X)^{D}$. This completes the proof
by~\cite[Theorem 3.1]{Y}.\end{proof}

A bounded linear operator $T\in \mathcal{L}(X)$ is Fredholm operator if $dimN(T)$ and $codimR(T)$ are finite, where $N(T)$ and $R(T)$ are the null space and the range of $T$ respectively. For each nonnegative integer $n$ define $T_{|n|}$ to be the restriction of $T$ to $R(T^n)$. If for some $n$, $R(T^n)$ is closed and $T_{|n|}$ is a Fredholm operator then $T$ is called a B-Fredholm operator. The B-Fredholm of $T$ are defined by
$$\sigma_{BF}(T)=\{ \lambda\in {\Bbb C}~|~T-\lambda I~\mbox{is not a B-Fredholm operator}\}.$$

\begin{cor} Let $A,B,C\in \mathcal{L}(X)$ such that $$(ABA)B=(ACA)B, (ABA)C=(ACA)C,$$ then $$\sigma_{BF}(AC)=\sigma_{BF}(BA).$$\end{cor}
\begin{proof} Let $\pi: \mathcal{L}(X)\to \mathcal{L}(X)/F(X)$ be the canonical map and $F(X)$ be the ideal of finite rank operators in $\mathcal{L}(X)$. As is well known, $T\in \mathcal{L}(X)$ is $B$-Fredholm if and only if $\pi(T)$ has Drazin inverse. By assumption, we have
$$\begin{array}{c}
\pi(A)\pi(B)\pi(A)\pi(B)=\pi(A)\pi(C)\pi(A)\pi(B),\\
\pi(A)\pi(B)\pi(A)\pi(C)=\pi(A)\pi(C)\pi(A)\pi(C).
\end{array}$$ The corollary is therefore established by Theorem 4.4.\end{proof}

\end{document}